\documentclass[12pt]{article}

\usepackage{amsmath, amsfonts, amsthm, mathtools, amscd}
\usepackage{mathrsfs} 
\usepackage{dsfont}
\usepackage[mathscr]{eucal}
\usepackage{graphicx}
\usepackage{csquotes}
\usepackage{enumitem}
\usepackage{xcolor}

\usepackage[T1]{fontenc}
\usepackage{newtxtext, newtxmath} 

\usepackage[english]{babel}
\usepackage{microtype}
\usepackage{setspace}
\usepackage[a4paper, textwidth=13.5cm]{geometry}
\usepackage{titling}

\pretitle{\begin{center}\LARGE\bfseries}
	\posttitle{\par\end{center}\vskip 0.5em}
\preauthor{\begin{center}\large}
	\postauthor{\end{center}}
\predate{\begin{center}\normalsize}
	\postdate{\end{center}}

\usepackage{titlesec}

\titleformat{\section}{\large\bfseries}{\thesection}{1em}{}
\titleformat{\subsection}{\normalsize\bfseries}{\thesubsection}{1em}{}

\newcommand{\abs}[1]{\left|#1\right|}

\newcommand{\dd}{\mathrm{d}}
\newcommand{\C}{{\mathbb C}}

\newcommand{\Mod}[1]{(\mathrm{mod}\,#1)}

\renewcommand{\le}{\leqslant}
\renewcommand{\ge}{\geqslant}

\usepackage{hyperref} 
\hypersetup{
	colorlinks=true,      
	linkcolor=blue,        
	urlcolor=blue,         
	citecolor=blue        
}

\usepackage[capitalize]{cleveref}

\usepackage{aliascnt}

\newtheorem{thm}{Theorem}
\newtheorem*{thm*}{Theorem}

\newcommand{\newaliasedtheorem}[3]{%
    \newaliascnt{#1}{thm}%
    \newtheorem{#1}[#1]{#3}%
    \aliascntresetthe{#1}%
    \crefname{#1}{#2}{#2s}%
    \Crefname{#1}{#3}{#3s}%
}

\newaliasedtheorem{lem}{lemma}{Lemma}
\newaliasedtheorem{clm}{claim}{Claim}
\newaliasedtheorem{prop}{proposition}{Proposition}
\newaliasedtheorem{cor}{corollary}{Corollary}
\newaliasedtheorem{axm}{axiom}{Axiom}
\newaliasedtheorem{fact}{fact}{Fact}
\newaliasedtheorem{conj}{conjecture}{Conjecture}

\theoremstyle{remark}
\newaliasedtheorem{rem}{remark}{Remark}
\newaliasedtheorem{rmk}{remark}{Remark}
\newaliasedtheorem{note}{note}{Note}
\newaliasedtheorem{qs}{question}{Question}
\newaliasedtheorem{obs}{observation}{Observation}

\theoremstyle{definition}
\newaliasedtheorem{defn}{definition}{Definition}

\title{On a series related to the M\"obius function}
\author{Deep Thakur}
\date{}

\begin{document}
\maketitle	
\bgroup
\let\thefootnote\relax\footnotetext{The Institute of Mathematical Sciences, A CI of Homi Bhabha National Institute, 
CIT Campus, Taramani, Chennai 600113, India. Email: deepthakur@imsc.res.in}
\egroup

\begin{abstract}
For a natural number $n>1$, let $P^-(n)$ denote its smallest prime divisor. Alladi and Johnson proved that $\sum_{P^{-}(n)\equiv l \Mod{k}} \frac{\mu(n)\omega(n)}{n} = 0$ for $(l,k)=1$. This was subsequently generalized by Tenenbaum, he proved that $\sum_{P^{-}(n) \in \mathcal{P}} \frac{\mu(n)\omega(n)}{n} = 0$ when $\mathcal{P}$ is a set of primes with a natural density. Moreover if we have $\sum_{\substack{p\le x \\ p\in\mathcal{P}}}\log p = \delta x +O\left( x\exp(-(\log x)^{\tau})\right)$ for some $\delta$ and with the parameter $\tau \in(0,3/5)$,\, then Tenenbaum showed that the partial sums of the above series up to $X$ are $O\left((\log\log X)^{1/\tau}/\log X\right)$. In this article, we prove that these partial sums are in fact equal to $C(\mathcal{P})/\log X + O\left((\log\log X)^{2/\tau}/(\log X)^2\right)$ with an explicit constant $C(\mathcal{P})$. We also investigate the behaviour of these constants in the classical case when $\mathcal{P}=\mathcal{P}(k,l)$ is the set of primes in the arithmetic progression $l\Mod{k}$ where $(l,k)=1$. In particular, we show under GRH that the number of reduced classes $l\Mod{k}$ for which $C(\mathcal{P}(k,l))< 0$ is $\varphi(k)(1- O(1/\log k))$ as $k\to\infty$.  
\end{abstract}

\section{Introduction}

Throughout, the letters $p$ and $q$ denote prime numbers. For a natural number $n>1$, let $P^-(n)$ denote its smallest prime divisor, and set $P^-(1)=\infty$. Furthermore, let $\mu$ and $\omega$ denote the M\"obius function and the number of distinct prime divisors, respectively. For natural numbers $n$ with $P^-(n)$ lying in an arithmetic progression $l\Mod{k}$, where $(l,k)=1$, it was established by Alladi \cite{alladi1977} that 
$$
\sum_{\substack{n\ge 2 \\ P^{-}(n)\equiv l\Mod{k}}} \frac{\mu(n)}{n} =-\frac{1}{\varphi(k)}.
$$
Alladi's proof was based on a duality identity between the smallest and the largest prime factors. By appealing to a higher-order duality identity between the smallest and the second-largest prime factors of natural numbers, Alladi and Johnson \cite{alladi2026} proved that
$$
\sum_{P^{-}(n)\equiv l \Mod{k}} \frac{\mu(n)\omega(n)}{n} = 0.
$$ 
Moreover, they showed that the estimate
$$
\sum_{\substack{n\le X \\ P^{-}(n)\equiv l \Mod{k}}} \frac{\mu(n)\omega(n)}{n} = O\left( \frac{(\log\log X)^{5/2}}{\sqrt{\log X}}\right)
$$ 
holds for all large $X>0$.

 We say a set of primes $\mathcal{P}$ has natural density $\delta\in [0,1]$ if we have that
 $$
 \lim_{x\to\infty} \varepsilon_{\mathcal{P}}(x) = 0,
 $$
where 
$$
\varepsilon_{\mathcal{P}}(x) = \frac1x\sum_{\substack{p\le x \\ p\in\mathcal{P}}}\log p - \delta.
$$
If $\mathcal{P}$ is a set of primes with a natural density, then Tenenbaum \cite{tenenbaum} shows that 
$$
\sum_{P^-(n)\in\mathcal{P}} \frac{\mu(n)\omega(n)}{n} = 0
$$
with an effective estimate for the rate of convergence in terms of $\varepsilon_{\mathcal{P}}$.
Moreover if $\varepsilon_{\mathcal{P}}(x) \ll \exp(-(\log x)^{\tau})$ with $0<\tau < 3/5$, then Tenenbaum \cite{tenenbaum} shows that the estimate
$$
\sum_{\substack{n\le X \\ P^-(n)\in\mathcal{P}}} \frac{\mu(n)\omega(n)}{n} = O \left( \frac{(\log\log X)^{1/\tau}}{\log X}\right)
$$ 
holds for all large $X>0$.

In this article, we provide an asymptotic formula for these partial sums. We prove the following.

\begin{thm} \label{thm:thm}

Let $\mathcal{P}$ be a set of primes with natural density $\delta\in [0,1]$ such that $\varepsilon_{\mathcal{P}}(x) \ll \exp(-c(\log x)^{\tau})$ with $0<\tau< 3/5$. Then, for all sufficiently large $X>0$ we have that
$$
\sum_{\substack{n\le X\\ P^-(n)\in\mathcal{P}}}\frac{\mu(n)\omega(n)}{n}=\frac{C(\mathcal{P})}{\log X}+O\left(\frac{(\log\log X)^{2/\tau}}{(\log X)^2}\right),
$$
where for all $x\ge 1$ we put $\zeta(s,x) = \prod_{q\le x}\left(1 - q^{-s}\right)^{-1}$ and
$$
C(\mathcal{P}) = -\delta+\sum_{p} \left( \mathds{1}_{p\in\mathcal{P}}-\delta\right) \frac{\zeta(1,p)}{p}.
$$

\end{thm}

The following corollary is then a consequence of the prime number theorem for arithmetic progressions which follows from the Vinogradov--Korobov zero-free region (see \cite{korobov1958,vinogradov1958}).

\begin{cor}
Let $l,k$ be integers satisfying $1\le l\le k$ and $(l,k)=1$. Then, for any $\varepsilon>0$ and for all sufficiently large $X>0$ we have that
$$
\sum_{\substack{n\le X \\ P^-(n)\equiv l\Mod{k}}} \frac{\mu(n)\omega(n)}{n} = \frac{C(k,l)}{\log X} +O_{k,\varepsilon}\left( \frac{(\log\log X)^{10/3+\varepsilon}}{(\log X)^2} \right)
$$
where 
$$
C(k,l) = -\frac{1}{\varphi(k)} +\sum_{p} \left( \mathds{1}_{p\equiv l\Mod{k}}-\frac{1}{\varphi(k)}\right) \frac{\zeta(1,p)}{p}.
$$
\end{cor}

 Hence, establishing the correct order of magnitude for these partial sums is tied to the non-vanishing of the constants $C(k,l)$. We discuss the behaviour of these constants in \Cref{sec:constants}. In particular, we show under GRH that the number of reduced classes $l\Mod{k}$ for which $C(k,l)< 0$ is $\varphi(k)(1- O(1/\log k))$ as $k\to\infty$ (see \Cref{thm:grh}).

\section{Proof of \Cref{thm:thm}}

For a prime $p$ we define 
$$
V_p(X) = \sum_{\substack{n\le X\\ P^-(n)=p }} \frac{\mu(n)\omega(n)}{n}
\quad\text{and}\quad
\beta(p)=\mathds{1}_{p\in \mathcal{P}}-\delta.
$$
It follows that
$$
\sum_{\substack{n\le X \\ P^-(n)\in\mathcal{P}}} \frac{\mu(n)\omega(n)}{n} = \sum_{\substack{p\le X\\ p\in\mathcal{P}}}V_p(X) = \delta \sum_{n \le X} \frac{\mu(n)\omega(n)}{n} + \sum_{p\le X} \beta(p)V_p(X).
$$
For a fixed prime $p$, Tenenbaum \cite[Proposition 3.2]{tenenbaum} proves that
$$
V_p(X)  \sim \frac{\zeta(1,p)}{p\log X}\quad(X\to\infty).
$$
We will show that the following uniform version of the above asymptotic formula holds.

\begin{prop}\label{prop:prop}

For all sufficiently large $X>0$ we have 
$$
V_p(X)  = \frac{\zeta(1,p)}{p\log X} +O\left(\frac{(\log p)^2}{p(\log X)^2}\right),
$$
uniformly for $p\le \exp(\sqrt{\log X})$.

\end{prop}

From \Cref{prop:prop} we may deduce \Cref{thm:thm} as follows.

\begin{proof}[Proof of \Cref{thm:thm}]
 We have that (see \Cref{rem:unrestricted} at the end of \Cref{sec:proof})
$$
 \sum_{n \le X} \frac{\mu(n)\omega(n)}{n} = -\frac{1}{\log X} + O\left(\frac{1}{(\log X)^2}\right).
$$
Let $1\le Y \le \exp(\sqrt{\log X})$ be a parameter to be chosen later. Then, we have 
$$
\sum_{p \le X} \beta(p)V_p(X) = \sum_{p \le Y} \beta(p)V_p(X) + \sum_{Y < p \le X} \beta(p)V_p(X).
$$
It follows from \Cref{prop:prop} that 
\begin{align*}
\sum_{p \le Y} \beta(p)V_p(X) &= \frac{1}{\log X}\sum_{p \le Y} \frac{\beta(p)\zeta(1,p)}{p} + O\left(\frac{(\log Y)^2}{(\log X)^2}\right) 
\\
&= \frac{1}{\log X}\sum_{p} \frac{\beta(p)\zeta(1,p)}{p}\\
&\qquad + O\left(\frac{(\log Y)^2}{(\log X)^2}+\frac{\exp(-c(\log Y)^{\tau})}{\log X}\right) ,
\end{align*}
where we have used the estimates
$$
\zeta(1,p) = e^\gamma \log p +O\left(\exp(-c(\log p)^{3/5}(\log\log p)^{-1/5})\right)
$$
and 
\begin{gather*}
\sum_{X_1 < p \le X_2}  \frac{\beta(p)\log p}{p} \ll \exp(-c(\log X_1)^{\tau}),\\
\sum_{X_1 < p \le X_2}  \frac{\beta(p)}{p} \ll \exp(-c(\log X_1)^{\tau}),
\end{gather*}
together with $\tau < 3/5$. The last two estimates follow by partial summation from
$$
\sum_{p\le X} \beta(p) \log p = X\varepsilon_{\mathcal{P}}(X)+\delta\left(X-\sum_{p\le X}\log p\right) \ll X \exp(-c(\log X)^{\tau}),
$$
where $\sum_{p\le X}\log p=X+O\left(X\exp(-c(\log X)^{3/5}(\log\log X)^{-1/5})\right)$ by the prime number theorem which follows from the Vinogradov--Korobov zero-free region (see \cite{korobov1958,vinogradov1958}). And for the tail we estimate as follows,
\begin{align*}
 \sum_{Y < p \le X} \beta(p)V_p(X) &= -\sum_{Y < p \le X} \frac{\beta(p)}{p} \sum_{\substack{n\le X/p \\ P^-(n)>p }}\frac{\mu(n)(\omega(n)+1)}{n} 
 \\
 & = - \sum_{n < X/Y} \frac{\mu(n)(\omega(n)+1)}{n} \sum_{\substack{Y < p \le X/n \\ p < P^-(n)}}  \frac{\beta(p)}{p} 
 \\
 & \ll \exp(-c(\log Y)^{\tau}) \sum_{n < X/Y} \frac{\omega(n)+1}{n}
  \\
 & \ll (\log\log X)(\log X)\exp(-c(\log Y)^{\tau}).
\end{align*}
Choosing $Y=\exp\left((A\log\log X)^{1/\tau}\right)$ with a sufficiently large constant $A$, we see that 
$$
\sum_{p \le X} \beta(p)V_p(X) = \frac{1}{\log X}\sum_{p} \frac{\beta(p)\zeta(1,p)}{p} + O\left( \frac{(\log\log X)^{2/\tau}}{(\log X)^2} \right).
$$ 
Hence,
\begin{align*}
\sum_{\substack{n\le X \\ P^-(n)\in\mathcal{P}}} \frac{\mu(n)\omega(n)}{n} &= \left( -\delta +\sum_{p} \left( \mathds{1}_{p\in\mathcal{P}}-\delta\right) \frac{\zeta(1,p)}{p} \right)\frac{1}{\log X}\\
&\qquad + O\left( \frac{(\log\log X)^{2/\tau}}{(\log X)^2} \right) .
\end{align*}
\end{proof}

In the next section we prove \Cref{prop:prop} as an application of the Selberg--Delange method.

\section{Proof of  \Cref{prop:prop}}

\label{sec:proof}
We write 
$$
V_p(X) = -\frac{1}{p} \sum_{\substack{n\le X/p \\ P^-(n)>p}} \frac{\mu(n)(\omega(n)+1)}{n}
$$
and consider the Dirichlet series 
$$
F_p(s) = -\frac{1}{p}  \sum_{\substack{n\ge 1 \\ P^-(n) > p}} \frac{\mu(n)(\omega(n)+1)}{n^s}.
$$
An easy calculation shows that 
$$
F_p(s) = \frac{\zeta(s,p)}{p\zeta(s)} \left( \sum_{q>p} \frac{1}{q^s-1} - 1  \right).
$$
For $\sigma>1$, we see that
$$
 \sum_{q} \frac{1}{q^s-1} - \log \zeta(s) = \sum_{q} \sum_{j=2}^\infty \frac{j-1}{jq^{js}} = H(s),
$$
where the function $H(s)$ is bounded and holomorphic in the half-plane $\sigma \ge 3/4$. Thus, 
$$
\frac{F_p(s+1)}{s} = \frac{\zeta(s+1,p)}{ps\zeta(s+1)} \left( -\log (s) +R_p(s)\right),
$$
where $\log$ denotes the principal branch of the logarithm on $\C\setminus (-\infty,0]$, and 
$$
R_p(s) = \log (s\zeta(s+1)) + H(s+1) - 1 -  \sum_{q\le p} \frac{1}{q^{s+1}-1} .
$$
Let $c_0>0$ be a sufficiently small absolute constant such that $\zeta(s)$ has no zeros in the region 
$$
\sigma \ge 1 - \frac{c_0}{\log(\abs{t} +2)}.
$$
Let 
$$
\kappa_0 = \frac{1}{\log (X/p)},~ T = \exp(\sqrt{c_0\log X})\text{  and }\kappa_1 = \frac{c_0}{\log (T+2)}.
$$ 
An application of the effective Perron formula (see \cite[Theorem II.2.3]{tenenbaumbook}) shows that 
$$
V_p(X) = \frac{1}{2\pi i} \int_{\kappa_0 - i T}^{\kappa_0 + i T} \frac{F_p(s+1)}{s}\left(\frac{X}{p}\right)^s ds + O\left( \frac{(\log X)^2}{pT} \right).
$$

We now evaluate this integral by replacing the integration segment by a truncated Hankel contour around $[-\kappa_1, 0]$, concatenated with two vertical segments $[-\kappa_1 , -\kappa_1 \pm iT]$ and two horizontal segments $[-\kappa_1 \pm iT , \kappa_0 \pm iT]$. We have the classical zero-free region bounds (see \cite[Theorems~II.3.21 and II.3.22]{tenenbaumbook})
$$
\frac{1}{\log(\abs{t}+2)} \ll \abs{\zeta(s)} \text{ and } \abs{\log\zeta(s)} \ll \log\log(\abs{t}+4),
$$
uniformly in the region $\sigma \ge 1 - \frac{c_0}{\log(\abs{t} +2)}$ and $\abs{t} \ge 2$. 

On the horizontal segments $[-\kappa_1 \pm iT , \kappa_0 \pm iT]$, we have the estimate 
$$
\frac{F_p(s+1)}{s}\left(\frac{X}{p}\right)^s \ll \frac{(\log p)^2(\log T)^2}{pT},
$$
uniformly for all $p\le \exp(\sqrt{\log X})$, since there $\abs{(X/p)^s}\le e$ and $\abs s\ge T$. On the vertical segments $[-\kappa_1 , -\kappa_1 \pm iT]$, we have the estimate
$$
\frac{F_p(s+1)}{s}\left(\frac{X}{p}\right)^s \ll \frac{1}{p(\abs{t}+2)} \exp\left(-\sqrt{\tfrac{c_0}{2} \log X}\right),
$$
uniformly for all $p\le \exp(\sqrt{\log X})$, since there $\abs{(X/p)^s}=(X/p)^{-\kappa_1}\ll\exp(-\sqrt{c_0\log X})$ and the remaining factors are $\ll(\log X)^2/(p(\abs t+2))$. Thus, the contribution of the horizontal and the vertical segments to the integral is 
$$
O\left(\frac{1}{p}\exp\left(-\sqrt{\tfrac{c_0}{4} \log X}\right)\right).
$$
Since $\frac{\zeta(s+1,p)R_p(s)}{ps\zeta(s+1)}$ is holomorphic in a neighbourhood of the disc $\abs{s}\le\kappa_1$, it follows that 
$$
\frac{1}{2\pi i} \int_{\mathcal{H}} \frac{\zeta(s+1,p)R_p(s)}{ps\zeta(s+1)} \left(\frac{X}{p}\right)^s ds = 0
$$
where $\mathcal{H}$ is the Hankel contour around $[-\kappa_1, 0]$.
Hence,
$$
V_p(X) = -\frac{1}{2\pi i} \int_{\mathcal{H}} \frac{\zeta(s+1,p)\log (s)}{ps\zeta(s+1)} \left(\frac{X}{p}\right)^s ds + O\left(\frac{1}{p}\exp\left(-\sqrt{\tfrac{c_0}{4} \log X}\right)\right).
$$
On the upper edge ($re^{i\pi}$ for $0\le r\le \kappa_1$) of $\mathcal{H}$, we have $-\log s = - \log r - i\pi$, and on the lower edge ($re^{-i\pi}$ for $0\le r\le \kappa_1$) we have $-\log s = - \log r + i\pi$, while the contribution of a circle $\abs s=\varepsilon$ around $0$ tends to $0$ with $\varepsilon$. Thus, 
$$
-\frac{1}{2\pi i} \int_{\mathcal{H}} \frac{\zeta(s+1,p)\log (s)}{ps\zeta(s+1)} \left(\frac{X}{p}\right)^s ds = \int_{-\kappa_1}^{0} \frac{\zeta(\sigma+1,p)}{p\sigma\zeta(\sigma+1)} \left(\frac{X}{p}\right)^{\sigma} d\sigma.
$$
For $-\kappa_1 \le \sigma \le 0$, we have the estimate 
$$
\frac{\zeta(\sigma+1,p)}{p\sigma\zeta(\sigma+1)} = \frac{\zeta(1,p)}{p} + O\left( \frac{\abs{\sigma}(\log p)^2}{p} \right)
$$
uniformly for all $p\le \exp(\sqrt{\log X})$. This follows from $\frac{1}{\sigma\zeta(\sigma+1)}=1+O(\abs\sigma)$ and from $\frac{\dd}{\dd\sigma}\log\zeta(\sigma+1,p)=-\sum_{q\le p}\frac{\log q}{q^{\sigma+1}-1}\ll\log p$. Substituting this into the integral, we obtain

\begin{align*}
\int_{-\kappa_1}^{0} \frac{\zeta(\sigma+1,p)}{p\sigma\zeta(\sigma+1)} \left(\frac{X}{p}\right)^{\sigma} d\sigma
&= \frac{\zeta(1,p)}{p} \int_{-\kappa_1}^{0} e^{\sigma \log(X/p)} d\sigma \\
&\qquad + O\left(\frac{(\log p)^2}{p} \int_{-\kappa_1}^{0} |\sigma| e^{\sigma \log(X/p)} d\sigma\right) \\
&= \frac{\zeta(1,p)}{p \log(X/p)} \left(1 - \left(\frac{X}{p}\right)^{-\kappa_1}\right) \\
&\qquad + O\left(\frac{(\log p)^2}{p (\log(X/p))^2}\right) \\
&= \frac{\zeta(1, p)}{p \log(X/p)} + O\left(\frac{(\log p)^2}{p (\log X)^2}\right) \\
&= \frac{\zeta(1, p)}{p \log X} + O\left(\frac{(\log p)^2}{p (\log X)^2}\right).
\end{align*}
This finishes the proof of \Cref{prop:prop}.

\begin{rem}\label{rem:unrestricted}
By an identical contour integration applied to the Dirichlet series 
$$
F(s) = \sum_{n=1}^{\infty} \frac{\mu(n)\omega(n)}{n^s} = -\frac{1}{\zeta(s)} \sum_{q} \frac{1}{q^s - 1},
$$
one obtains
$$
\sum_{n \le X} \frac{\mu(n)\omega(n)}{n} = -\frac{1}{\log X} + O\left(\frac{1}{(\log X)^2}\right).
$$
\end{rem}

\section{The constants $C(k,l)$}\label{sec:constants}

We begin with an elementary identity.

\begin{lem}\label{lem:lem}
Let $k\ge1$ and $(l,k)=1$. Then for every $X\ge1$, we have
$$
C(k,l)= \sum_{\substack{p\le X\\ p\equiv l\Mod{k}}} \frac{\zeta(1,p)}{p} -\frac{\zeta(1,X)}{\varphi(k)}+ \sum_{p>X}\left(\mathds{1}_{p\equiv l\Mod{k}}-\frac1{\varphi(k)}\right) \frac{\zeta(1,p)}{p}.
$$
\end{lem}
 
\begin{proof}
For $Y\ge 1$, we note that
\begin{equation}\label{eq:tele}
1+\sum_{p\le Y} \frac{\zeta(1,p)}{p} = \zeta(1,Y).
\end{equation}
Hence, for $Y> X$ we have
\begin{align*}
&-\frac{1}{\varphi(k)} +\sum_{p \le Y} \left( \mathds{1}_{p\equiv l\Mod{k}}-\frac{1}{\varphi(k)}\right) \frac{\zeta(1,p)}{p}
\\
&\quad=
\sum_{\substack{p\le X \\ p\equiv l\Mod{k}}} \frac{\zeta(1,p)}{p} -\frac{\zeta(1,X)}{\varphi(k)}+\sum_{X< p \le Y}\left(\mathds{1}_{p\equiv l\Mod{k}}-\frac1{\varphi(k)}\right) \frac{\zeta(1,p)}{p}.
\end{align*}
The lemma then follows by letting $Y\to\infty$.
\end{proof}

For the rest of this section we assume the generalized Riemann hypothesis (GRH). Then
$$
\theta(x,\chi)=\sum_{p\le x}\chi(p)\log p\ll x^{1/2}(\log kx)^2
$$
for every non-trivial character $\chi\Mod{k}$ and $x\ge2$ (see \cite[Chapter 20]{Davenport1980book}). By partial summation, we get
$$
\sum_{y\le p\le Y}\frac{\chi(p)}{p}\ll\frac{\log ky}{y^{1/2}}
$$
uniformly for $Y\ge y\ge k$. Moreover, for $Y>X\ge k$ we get
\begin{align*}
\sum_{X<p\le Y}\chi(p)\frac{\zeta(1,p)}{p} &= \zeta(1,X)\sum_{X<p\le Y}\frac{\chi(p)}{p}+\sum_{X<q\le Y}\frac{\zeta(1,q)}{q}\sum_{q\le p\le Y}\frac{\chi(p)}{p}
\\
&\ll \frac{(\log X)(\log kX)}{X^{1/2}}+\sum_{q>X}\frac{(\log q)(\log kq)}{q^{3/2}}
\\
&\ll\frac{(\log kX)^2}{X^{1/2}}.
\end{align*}
Letting $Y\to\infty$, we obtain
\begin{equation}\label{eq:T}
\sum_{p>X}\chi(p)\frac{\zeta(1,p)}{p} \ll \frac{(\log kX)^2}{X^{1/2}} \qquad( X\ge k).
\end{equation}

\begin{thm}\label{thm:grh}
	Assume GRH. Then for all sufficiently large integer $k>0$ we have
	$$
	\#\bigl\{1\le l\le k:\ (l,k)=1,\ C(k,l)\ge0\bigr\}\ll\frac{\varphi(k)}{\log k}
	$$
	and 
	$$
	\#\bigl\{1\le l\le k:\ (l,k)=1,\ C(k,l)>0\bigr\}\gg\frac{\varphi(k)}{\log k}.
	$$
	In particular,
	$$
	\#\bigl\{1\le l\le k:\ (l,k)=1,\ C(k,l)<0\bigr\}=\varphi(k)\left(1-O\Bigl(\frac{1}{\log k}\Bigr)\right).
	$$
\end{thm}

\begin{proof}
	Throughout, $k$ is large, $l$ runs over the integers $1\le l\le k$ with $(l,k)=1$, and $\mathbb{P}$ denotes the set of primes. We may fix an absolute constant $A\ge1$ such that
	$$
	\log x\le\zeta(1,x)\le A\log x\qquad(x\ge2),
	$$
	where the lower bound follows from $\zeta(1,x)\ge\sum_{n\le x}1/n$. Set $X=k(\log k)^4$.
	
	The only prime $p\le k$ with $p\equiv l\Mod{k}$ is $l$ itself, when $l$ is prime. Hence \Cref{lem:lem}, applied with $k$ in place of $X$, gives
	$$
	C(k,l)=\mathds{1}_{l\in\mathbb{P}}\frac{\zeta(1,l)}{l}-\frac{\zeta(1,k)}{\varphi(k)}+S(k,l)+T(k,l),
	$$
	where
	$$
	S(k,l)=\sum_{k<p\le X}\Bigl(\mathds{1}_{p\equiv l\Mod{k}}-\frac{1}{\varphi(k)}\Bigr)\frac{\zeta(1,p)}{p}
	$$
	and
	$$
	T(k,l)=\sum_{p>X}\Bigl(\mathds{1}_{p\equiv l\Mod{k}}-\frac{1}{\varphi(k)}\Bigr)\frac{\zeta(1,p)}{p}.
	$$
	Since $\zeta(1,k)\ge\log k$, if $C(k,l)\ge0$, then at least one of the following holds:
	\begin{equation}\label{eq:PNT}
		\mathds{1}_{l\in\mathbb{P}}\frac{\zeta(1,l)}{l}\ge\frac{\log k}{3\varphi(k)},
	\end{equation}
	\begin{equation}\label{eq:BT}
		\abs{S(k,l)}\ge\frac{\log k}{3\varphi(k)},
	\end{equation}
	\begin{equation}\label{eq:GRH}
		\abs{T(k,l)}\ge\frac{\log k}{3\varphi(k)}.
	\end{equation}
	We bound the number of $l$ satisfying each of these conditions.
	
	\medskip\noindent\emph{The condition \eqref{eq:PNT}.}
	If \eqref{eq:PNT} holds, then $l$ is a prime with $l\le3A\varphi(k)$. Since $\varphi(k)\gg k/\log\log k$, the number of such $l$ is at most $\pi(3A\varphi(k))\ll\varphi(k)/\log k$.
	
	\medskip\noindent\emph{The condition \eqref{eq:BT}.}
	The number of $l$ satisfying \eqref{eq:BT} is at most
	$$
	\frac{9\varphi(k)^2}{(\log k)^2}\sum_{l}\abs{S(k,l)}^2.
	$$
	By orthogonality of Dirichlet characters we have,
	$$
	S(k,l)=\frac{1}{\varphi(k)}\sum_{\chi\neq\chi_0\Mod{k}}\overline{\chi(l)}\,S(\chi),
	\qquad\text{where}\quad
	S(\chi)=\sum_{k<p\le X}\frac{\chi(p)\zeta(1,p)}{p}.
	$$
	Since $\zeta(1,p)\le A\log X\ll\log k$ for $p\le X$, we get
	\begin{align*}
		\sum_{l}\abs{S(k,l)}^2
		&=\frac{1}{\varphi(k)}\sum_{\chi\neq\chi_0\Mod{k}}\abs{S(\chi)}^2
		\le\frac{1}{\varphi(k)}\sum_{\chi\Mod{k}}\abs{S(\chi)}^2\\
		&=\sum_{\substack{k<p_1,p_2\le X\\ p_1\equiv p_2\Mod{k}}}\frac{\zeta(1,p_1)\zeta(1,p_2)}{p_1p_2}\\
		&\ll\frac{\log k}{k}+(\log k)^2\sum_{\substack{k<p_1<p_2\le X\\ p_1\equiv p_2\Mod{k}}}\frac{1}{p_1p_2}.
	\end{align*}
	 In the remaining sum we write $p_2=p_1+mk$ with $1\le m<(\log k)^4$. By Selberg's upper bound sieve (see \cite[Theorem~5.7]{HR74}), we have
	$$
	\#\bigl\{x<p\le2x:\ p+mk\in\mathbb{P}\bigr\}\ll\frac{x\log\log k}{(\log k)^2}
	$$
	uniformly for $x\ge k$ and $1\le m\le(\log k)^4$. Splitting $(k,X]$ into dyadic intervals, we obtain
	\begin{align*}
		\sum_{\substack{k<p_1<p_2\le X\\ p_1\equiv p_2\Mod{k}}}\frac{1}{p_1p_2}
		&\le\sum_{1\le m<(\log k)^4}\frac{1}{mk}\sum_{\substack{k<p\le X\\ p+mk\in\mathbb{P}}}\frac{1}{p}\\
		&\le\sum_{1\le m<(\log k)^4}\frac{1}{mk}\sum_{0\le j\le6\log\log k}\frac{1}{2^jk}\sum_{\substack{2^jk<p\le2^{j+1}k\\ p+mk\in\mathbb{P}}}1\\
		&\ll\frac{(\log\log k)^2}{k(\log k)^2}\sum_{1\le m<(\log k)^4}\frac{1}{m}
		\ll\frac{(\log\log k)^3}{k(\log k)^2}.
	\end{align*}
	Hence, the number of $l$ satisfying \eqref{eq:BT} is
	\[
	\ll\frac{\varphi(k)^2}{(\log k)^2}\cdot\frac{\log k}{k}\le\frac{\varphi(k)}{\log k}.
	\]
	
	\medskip\noindent\emph{The condition \eqref{eq:GRH}.}
	Similarly, $$
	T(k,l)=\frac{1}{\varphi(k)}\sum_{\chi\neq\chi_0}\overline{\chi(l)}\,T(\chi) 
	\quad\text{with}\quad
	T(\chi)=\sum_{p>X}\chi(p)\zeta(1,p)/p.
	$$ 
	It follows from the estimate \eqref{eq:T} that
	$$
	\sum_{l}\abs{T(k,l)}^2=\frac{1}{\varphi(k)}\sum_{\chi\neq\chi_0\Mod{k}}\abs{T(\chi)}^2\le\max_{\chi\neq\chi_0}\abs{T(\chi)}^2\ll\frac{(\log kX)^4}{X}\ll\frac{1}{k}.
	$$
	Hence, the number of $l$ satisfying \eqref{eq:GRH} is 
	$$
	\ll\frac{\varphi(k)^2}{(\log k)^2}\cdot\frac{1}{k}\le\frac{\varphi(k)}{(\log k)^2}.
	$$
	
	Combining the three estimates, we obtain
	$$
	\#\bigl\{1\le l\le k:\ (l,k)=1,\ C(k,l)\ge0\bigr\}\ll\frac{\varphi(k)}{\log k}.
	$$
	
	Let $L=\varphi(k)/(6A)$ and let $l\le L$ be a prime with $l\nmid k$. Keeping only the term $p=l$ in the first sum of \Cref{lem:lem}, we get
	$$
	C(k,l)\ge\frac{\zeta(1,l)}{l}-\frac{\zeta(1,X)}{\varphi(k)}+T(k,l)\ge\frac{\log l}{l}-\frac{A\log X}{\varphi(k)}+T(k,l).
	$$
	Since $k$ is large, we have $\log X\le2\log\varphi(k)$ and
	$$
	\frac{\log l}{l}\ge
	\frac{\log L}{L}=\frac{6A\bigl(\log\varphi(k)-\log(6A)\bigr)}{\varphi(k)}\ge\frac{3A\log\varphi(k)}{\varphi(k)}.
	$$
	Thus,
	$$
	C(k,l)\ge\frac{A\log\varphi(k)}{\varphi(k)}+T(k,l)\ge\frac{\log k}{2\varphi(k)}+T(k,l).
	$$
	Hence, if $C(k,l)\le 0$ then \eqref{eq:GRH} holds. The number of such primes $l$ is at least $\pi(L)-\omega(k)\gg\varphi(k)/\log k$, and by the bound for \eqref{eq:GRH}, at most $O(\varphi(k)/(\log k)^2)$ of them satisfy $C(k,l)\le 0$. Therefore
	$$
	\#\bigl\{1\le l\le k:\ (l,k)=1,\ C(k,l)> 0\bigr\}\gg\frac{\varphi(k)}{\log k},
	$$
	which completes the proof.
\end{proof}

\begin{rmk}
	Numerical computation of all $C(k,l)$ with $k\le 300$ suggests that the constants $C(k,l)$ are always non-zero.
\end{rmk}

 \section*{Acknowledgement}
 
The author would like to thank Professor Jean-Marc Deshouillers for fruitful discussions and encouragement. The author would also like to thank Professor Gérald Tenenbaum for his comments on an earlier version of this manuscript.

\end{document}